\documentclass[11pt,a4paper]{amsart}
\usepackage[T1]{fontenc}
\usepackage[utf8]{inputenc}
\usepackage{amsmath,amssymb}
\usepackage{enumerate}
\usepackage[a4paper,margin=2.8cm]{geometry}
\usepackage[colorlinks=true,linkcolor=blue,citecolor=blue]{hyperref}
\hypersetup{
  pdftitle={Milnor fibrations on rho-tubes and rho-spheres for real analytic map germs},
  pdfauthor={Gabriel França, Maico Ribeiro, Mateus de Melo},
  pdfsubject={Singularity theory; Milnor fibrations},
  pdfkeywords={Milnor fibration; rho-regularity; Milnor condition (b); real analytic map germs; open book decompositions}
}

\newtheorem{theorem}{Theorem}[section]
\newtheorem{proposition}[theorem]{Proposition}
\newtheorem{lemma}[theorem]{Lemma}
\newtheorem{corollary}[theorem]{Corollary}
\theoremstyle{definition}
\newtheorem{definition}[theorem]{Definition}
\newtheorem{example}[theorem]{Example}
\theoremstyle{remark}
\newtheorem{remark}[theorem]{Remark}

\DeclareMathOperator{\Sing}{Sing}
\DeclareMathOperator{\Disc}{Disc}
\DeclareMathOperator{\rank}{rank}
\DeclareMathOperator{\spann}{span}
\newcommand{\R}{\mathbb{R}}
\newcommand{\norm}[1]{\lVert #1\rVert}
\newcommand{\clo}[1]{\overline{#1}}

\title[Milnor fibrations on $\rho$-tubes and $\rho$-spheres]{Milnor fibrations on $\rho$-tubes and $\rho$-spheres for real analytic map germs}

\author{Gabriel Fran\c{c}a}
\address{Departamento de Matem\'atica, Universidade Federal do Esp\'irito Santo, Vit\'oria-ES, Brazil}
\email{gabriel.santos.78@ufes.br}

\author{Maico Ribeiro}
\address{Departamento de Matem\'atica, Universidade Federal do Esp\'irito Santo, Vit\'oria-ES, Brazil}
\email{maico.ribeiro@ufes.br}

\author{Mateus de Melo}
\address{Departamento de Matem\'atica, Universidade Federal do Esp\'irito Santo, Vit\'oria-ES, Brazil}
\email{mateus.melo@ufes.br}


\begin{document}

\begin{abstract}
For a real analytic map germ $G:(\R^m,0)\to(\R^p,0)$, we study Milnor fibrations obtained by replacing the squared Euclidean distance with an analytic control function $\rho$ defining the origin. We formulate the corresponding Milnor set and condition (b), derive criteria for tube and sphere fibrations, and examine their dependence on $\rho$. The family $G_\rho=(xz,\,yz\rho)$ provides explicit changes of regularity under elliptic controls. In particular, the germ
$G(x,y,z)=(xz,\,yz(10x^2+y^2+3z^2))$
admits neither the induced tube nor sphere fibration for the Euclidean control, while both fibrations exist for the adapted function $\rho=10x^2+y^2+3z^2$. Thus the control function can be essential to the local fibration structure.
\end{abstract}

\maketitle

\section{Introduction}\label{sec:intro}

Milnor's fibration theorem is a cornerstone for studying the local topology of a holomorphic function germ near a singular point \cite{Mil68}. It associates with each germ a locally trivial fibration, described either on a sphere or on a tube, making the existence and topology of this structure an invariant of the germ. 

For real analytic map germs
\[
G:(\mathbb{R}^m,0)\longrightarrow(\mathbb{R}^p,0),\qquad m\ge p\ge2,
\]
the existence of a Milnor-type fibration is no longer guaranteed, which has led to a broad theory (for a comprehensive survey, see \cite{Seade19}; see also \cite{ACT13,ART10,ARSRT19,ARSRT20,CMSS09,CSS10,CSS12,Le77,Loo84,Massey10,Pichon05,PichonSeade08,Quick22,RSV02,RS05,Seade97}). Regularity conditions such as Thom's $(a_G)$-regularity, the Milnor condition (b), $(d)$-regularity, and $\rho$-regularity were introduced to control the transversality needed in this setting \cite{ART10,ARSRT19,Bek91,CSS10,CSS12,Mather12,Thom69}.

These conditions are usually formulated with the squared Euclidean distance $\rho_E(x)=\norm{x}^2$. Its level sets provide the round spheres on which the local geometry is examined. There is, however, no intrinsic reason to privilege this particular function. Ellipsoids and more general level hypersurfaces may be better adapted to the geometry of a given germ, especially when its variables occur with different weights or degrees.

This point of view was developed in \cite{RST24}, where Ribeiro, Santamaria and da Silva considered $\rho$-regularity with respect to a general function defining the origin and raised the question of how the usual existence results should be reformulated. The purpose of the present paper is to pursue that question for both standard realizations of the Milnor fibration. We use the levels and sublevels of an analytic control function $\rho$ in the source, while retaining the Euclidean structure in the target, and formulate the corresponding transversality conditions for the tube and for the projection $\Psi_G=G/\norm G$ on the sphere.

A central role is played by the family
\[
G_\rho(x,y,z)=\bigl(xz,\,yz\rho(x,y,z)\bigr).
\]
It is simple enough to permit explicit calculations, but flexible enough to display a genuine dependence on the control function. Comparing different elliptic controls leads to a quantitative analysis of when transversality is preserved or lost. Related quasi-homogeneous examples show how resonances appear in the Euclidean setting, while Euler identities help locate the possible limiting tangencies.

The same family also makes it possible to compare tube and sphere phenomena within a single framework. On the sphere side, the relevant object is the angular projection $G/\norm G$, and the problem becomes one of transversality to the levels of $\rho$. The germ
\[
G(x,y,z)=\bigl(xz,\,yz(10x^2+y^2+3z^2)\bigr)
\]
illustrates the geometric distinction particularly clearly: the Euclidean control produces both radial and angular obstructions, whereas the adapted function $\rho=10x^2+y^2+3z^2$ is compatible with the geometry of the germ. This example emphasizes that the control function is not merely a convenient way to describe a neighborhood of the origin; it may be part of the fibration problem itself.

Our main results are structured along three principal lines:
\begin{itemize}
    \item \textbf{Comparison of Elliptic Controls:} For the family $G_\rho$, the validity of condition~(b) under a second control is determined by a sharp critical threshold formed by the ratio of their coefficients.
    \item \textbf{Resonance and Normalization:} For components with higher even powers, obstructions to Euclidean regularity stem from a resonance between exponents and coefficients, revealing this as an artifact of Euclidean normalization rather than an intrinsic property of the exponents.
    \item \textbf{Quasi-Homogeneous Constraints:} Euler relations restrict the accumulation of the Milnor set strictly to the singular locus, bounding where these topological obstructions can arise.
\end{itemize}
Together, these results explain the failure of Euclidean control for $G(x,y,z)=(xz,\,yz(10x^2+y^2+3z^2))$, where non-transversality persists over an open set of coefficients, while the adapted control successfully restores both tube and sphere fibrations.

The paper is organized as follows. Section \ref{sec:tube} introduces control functions, $\rho$-tubes, the associated Milnor set and condition (b), together with basic examples. Section \ref{sec:family} studies the family $G_\rho$ under changes of elliptic control. Section \ref{sec:qh} discusses quasi-homogeneous germs and the role of resonance and Euler identities. Section \ref{sec:sphere} develops the fibration on a $\rho$-sphere through the projection $G/\norm G$. Section \ref{sec:main} analyzes the explicit example above and the corresponding family of coefficient patterns.

\section{\texorpdfstring{$\rho$}{rho}-tubes and the Milnor condition (b)}\label{sec:tube}

\subsection{Control functions}\label{subsec:control}

Let \(U\subset\mathbb{R}^m\) be an open neighborhood of the origin. A smooth
function \(\rho\colon U\to[0,\infty)\) \emph{defines the origin} if:
\begin{enumerate}
\item[(a)] \(\rho^{-1}(0)=\{0\}\);
\item[(b)] there exists \(\varepsilon_0>0\) such that
\(\rho^{-1}([0,\varepsilon])\) is compact in \(U\) for every
\(0<\varepsilon\leq\varepsilon_0\);
\item[(c)] \(0\) is the only critical point of \(\rho\) in
\(\rho^{-1}([0,\varepsilon_0])\).
\end{enumerate}
Thus every \(\varepsilon\in(0,\varepsilon_0]\) is a regular value. Set
\[
B^m_{\rho,\varepsilon}:=\{\rho<\varepsilon\},\qquad
\overline{B}^m_{\rho,\varepsilon}:=\{\rho\leq\varepsilon\},\qquad
S^{m-1}_{\rho,\varepsilon}:=\rho^{-1}(\varepsilon).
\]
These are the open and closed \(\rho\)-balls and the \(\rho\)-sphere,
respectively. In particular, \(S^{m-1}_{\rho,\varepsilon}\) is a compact
smooth hypersurface and
\[
\partial\overline{B}^m_{\rho,\varepsilon}
=S^{m-1}_{\rho,\varepsilon}.
\]

The choice \(\rho(x)=\|x\|^2\) gives the usual Euclidean balls and spheres.
Positive-definite quadratic forms give ellipsoids, while sums
\(\sum_i x_i^{2a_i}\) give superellipsoids. More generally, every homogeneous
polynomial \(\rho\) that is positive away from the origin defines the origin,
since Euler's identity gives
\[
\langle\nabla\rho(x),x\rangle=(\deg\rho)\rho(x)>0
\qquad (x\neq0).
\]

\begin{theorem}[Topology of \(\rho\)-spheres]
Let \(\rho\colon U\to[0,\infty)\) define the origin. Then:
\begin{enumerate}[\rm(i)]
\item all \(S^{m-1}_{\rho,\varepsilon}\), with
\(0<\varepsilon\leq\varepsilon_0\), are smoothly diffeomorphic;
\item \(S^{m-1}_{\rho,\varepsilon}\) is homeomorphic to \(S^{m-1}\), and is
smoothly diffeomorphic to \(S^{m-1}\) whenever \(m\neq5\).
\end{enumerate}
\end{theorem}

\begin{proof}
On \(\overline{B}^m_{\rho,\varepsilon_0}\setminus\{0\}\), consider
\[
X=\frac{\nabla\rho}{\|\nabla\rho\|^2}.
\]
Since \(d\rho(X)=1\), its flow \(\phi_t\) satisfies
\[
\rho(\phi_t(x))=\rho(x)+t.
\]
Consequently,
\[
\phi_{\varepsilon_2-\varepsilon_1}\colon
S^{m-1}_{\rho,\varepsilon_1}\longrightarrow
S^{m-1}_{\rho,\varepsilon_2}
\]
is a diffeomorphism, proving (i).

The negative flow contracts \(\overline{B}^m_{\rho,\varepsilon}\) to the
origin:
\[
H(x,s)=
\begin{cases}
\phi_{-s\rho(x)}(x),&x\neq0,\\
0,&x=0.
\end{cases}
\]
Indeed, \(\rho(H(x,s))=(1-s)\rho(x)\), and condition~(b) ensures continuity
at \(0\). Hence \(\overline{B}^m_{\rho,\varepsilon}\) is contractible.

Lefschetz duality and the long exact sequence of the pair
$
\bigl(\overline{B}^m_{\rho,\varepsilon},
S^{m-1}_{\rho,\varepsilon}\bigr)
$
show that \(S^{m-1}_{\rho,\varepsilon}\) is an integral homology sphere.
For \(m\geq3\), removing the interior point \(0\) does not change the
fundamental group, and the punctured \(\rho\)-ball retracts onto its boundary.
Thus \(S^{m-1}_{\rho,\varepsilon}\) is simply connected and hence, by
Hurewicz and Whitehead, a homotopy sphere. The cases \(m\leq2\) are immediate.

 If \(m\geq6\), removing a small Euclidean ball around the
origin from \(\overline{B}^m_{\rho,\varepsilon}\) produces an
\(h\)-cobordism between \(S^{m-1}_{\rho,\varepsilon}\) and \(S^{m-1}\);
the \(h\)-cobordism theorem makes it a product. The cases \(m\leq4\) follow
from the classification of smooth manifolds of dimension at most \(3\).
The remaining case \(m=5\) is precisely the unresolved smooth
four-dimensional case.
\end{proof}

\subsection{The fibration on the \texorpdfstring{$\rho$}{rho}-tube}
Let $G:(\R^m,0)\to(\R^p,0)$, $m\ge p\ge2$, be a nonconstant real analytic map germ with representative $G:U\to\R^p$. We write $V_G:=G^{-1}(0)$, $\Sing G:=\{x\in U:\rank dG(x)<p\}$, and $\Disc(G):=G(\Sing G)$; the germ has \emph{isolated critical value} when $\Disc(G)\subseteq\{0\}$, in which case $\Sing G\subset V_G$. We use an inclusion rather than an equality so that submersions, for which $\Sing G=\emptyset$ and hence $\Disc(G)=\emptyset$, are not excluded. In the target $\R^p$ we always keep the Euclidean metric and write $B^p_\eta$, $\clo B^p_\eta$, $S^{p-1}_\eta$.

\begin{definition}\label{def:tube}
The germ $G$ \emph{admits a Milnor fibration on the $\rho$-tube} if there exists $\varepsilon_0>0$ such that for every $0<\varepsilon\le\varepsilon_0$ there exists $\eta=\eta(\varepsilon)$, $0<\eta\ll\varepsilon$, for which the restriction
\begin{equation}\label{eq:tubefib}
G\big|:\ \clo B^m_{\rho,\varepsilon}\cap G^{-1}\bigl(\clo B^p_\eta\setminus\{0\}\bigr)\ \longrightarrow\ \clo B^p_\eta\setminus\{0\}
\end{equation}
is a smooth locally trivial fibration whose diffeomorphism type does not depend on the (sufficiently small) choices of $\varepsilon$ and $\eta$.
\end{definition}

\begin{definition}\label{def:rhoreg}
$G$ is \emph{$\rho$-regular} if, for every sufficiently small $\varepsilon>0$, there is a neighborhood $N_\varepsilon$ of $S^{m-1}_{\rho,\varepsilon}\cap V_G$ in $S^{m-1}_{\rho,\varepsilon}$ such that, for every $x\in N_\varepsilon\setminus V_G$, the fiber $G^{-1}(G(x))$ meets $S^{m-1}_{\rho,\varepsilon}$ transversally at $x$.
\end{definition}

The following statement adapts the Euclidean existence theorem of \cite{RST24} to $\rho$-tubes; the proof is the classical Ehresmann argument.

\begin{proposition}\label{prop:tube}
If $G$ has isolated critical value and is $\rho$-regular, then $G$ admits a Milnor fibration on the $\rho$-tube.
\end{proposition}

\begin{proof}
Fix $\varepsilon>0$ small, so that $S^{m-1}_{\rho,\varepsilon}$ is a compact smooth hypersurface, and let $N_\varepsilon$ be as in Definition \ref{def:rhoreg}. Since $S^{m-1}_{\rho,\varepsilon}$ is compact and $G$ is continuous, there is $\eta>0$, $0<\eta\ll\varepsilon$, with
\[
S^{m-1}_{\rho,\varepsilon}\cap G^{-1}\bigl(\clo B^p_\eta\setminus\{0\}\bigr)\subset N_\varepsilon\setminus V_G .
\]
Write $E:=\clo B^m_{\rho,\varepsilon}\cap G^{-1}(\clo B^p_\eta\setminus\{0\})$ for the total space of \eqref{eq:tubefib}.

\emph{Properness.} The space $E$ is \emph{not} compact: the fibre over $0$ has been removed, and $\clo B^p_\eta\setminus\{0\}$ is not closed. Properness must therefore be verified directly, and it does hold. Let $C\subset\clo B^p_\eta\setminus\{0\}$ be compact. Then $C$ is closed in $\R^p$, so $G^{-1}(C)$ is closed in $U$ and
\[
E\cap G^{-1}(C)=\clo B^m_{\rho,\varepsilon}\cap G^{-1}(C)
\]
is a closed subset of the compact set $\clo B^m_{\rho,\varepsilon}$, hence compact.

\emph{The boundary faces.} $E$ is a manifold with corners whose boundary consists of the $\rho$-spherical face
\[
\partial_sE:=S^{m-1}_{\rho,\varepsilon}\cap G^{-1}\bigl(\clo B^p_\eta\setminus\{0\}\bigr)
\qquad\text{and}\qquad
\partial_oE:=\clo B^m_{\rho,\varepsilon}\cap G^{-1}\bigl(S^{p-1}_\eta\bigr),
\]
meeting along the corner $\partial_sE\cap\partial_oE=S^{m-1}_{\rho,\varepsilon}\cap G^{-1}(S^{p-1}_\eta)$. Since $\Disc(G)\subseteq\{0\}$, every point of $G^{-1}(\clo B^p_\eta\setminus\{0\})$ is a regular point of $G$ for $\eta$ small; hence $G$ is a submersion on the interior of $E$, the face $\partial_oE$ is a smooth manifold, and $G|_{\partial_oE}:\partial_oE\to S^{p-1}_\eta$ is a submersion. On $\partial_sE$ the choice of $\eta$ made above, together with Definition \ref{def:rhoreg}, gives that $G|_{\partial_sE}$ is a submersion as well. At a corner point both restrictions are submersions simultaneously, so the two face conditions are compatible, and Ehresmann's theorem for proper submersions of manifolds with corners (equivalently, after the standard smoothing of the corner) yields local triviality.

\emph{Surjectivity.} $E\ne\emptyset$; the image of the proper map $G|_E$ is closed in $\clo B^p_\eta\setminus\{0\}$, the image of a submersion is open, and $\clo B^p_\eta\setminus\{0\}$ is connected because $p\ge2$. Hence $G|_E$ is onto, and \eqref{eq:tubefib} is a locally trivial fibration with nonempty fibre.

\emph{Independence of $\varepsilon$ and $\eta$.} Given $0<\varepsilon'<\varepsilon$ with compatible $\eta'$, the corresponding total spaces are compared by the isotopy arguments of \cite{ACT13,ARSRT20} applied to nested pairs of regular levels of $\rho$: these produce a diffeomorphism commuting with the projections, so the diffeomorphism type of the fibre does not depend on the admissible choices.
\end{proof}

\subsection{The Milnor set and condition (b)}
From now on we assume in addition that $\rho$ is \emph{real analytic}, so that the Curve Selection Lemma is available.

\begin{definition}\label{def:milnorset}
The \emph{Milnor set of $G$ associated with $\rho$} is the set germ $M_\rho(G):=\Sing(G,\rho)$, that is, the critical set of the map $(G,\rho):\R^m\to\R^{p+1}$. Thus $x\in M_\rho(G)$ if and only if the gradients $\nabla G_1(x),\dots,\nabla G_p(x),\nabla\rho(x)$ are linearly dependent; in particular $\Sing G\subset M_\rho(G)$, and $M_\rho(G)$ is an analytic set. We say that $G$ satisfies the \emph{Milnor condition (b) associated with $\rho$} if
\begin{equation}\label{eq:condb}
\clo{M_\rho(G)\setminus V_G}\cap V_G\subseteq\{0\}
\end{equation}
as set germs at the origin.
\end{definition}

For $\rho=\rho_E$ this is the usual Milnor condition (b). Since $\nabla\rho\ne0$ off the origin, a point $x\in S^{m-1}_{\rho,\varepsilon}\setminus\Sing(\rho)$ lies outside $M_\rho(G)$ if and only if $G\big|_{S^{m-1}_{\rho,\varepsilon}}$ is a submersion at $x$; equivalently, if and only if the fiber $G^{-1}(G(x))$ meets $S^{m-1}_{\rho,\varepsilon}$ transversally at $x$ (when $x\notin\Sing G$).

\begin{lemma}\label{lem:btosub}
Let $G$ have isolated critical value and satisfy condition (b) associated with $\rho$. Then there is $\varepsilon_0>0$ such that for every $0<\varepsilon<\varepsilon_0$ there is $\eta>0$, $0<\eta\ll\varepsilon$, for which
\[
G\big|:\ S^{m-1}_{\rho,\varepsilon}\cap G^{-1}\bigl(\clo B^p_\eta\setminus\{0\}\bigr)\longrightarrow \clo B^p_\eta\setminus\{0\}
\]
is a smooth submersion. In particular $G$ is $\rho$-regular.
\end{lemma}

\begin{proof}
Fix $0<\varepsilon<\varepsilon_0$. The set $S^{m-1}_{\rho,\varepsilon}\cap V_G$ is compact and does not contain the origin, so \eqref{eq:condb} provides a neighborhood $N_\varepsilon$ of it inside $S^{m-1}_{\rho,\varepsilon}$ that is disjoint from $M_\rho(G)\setminus V_G$. By continuity of $G$ and compactness, there is $0<\eta\ll\varepsilon$ with $S^{m-1}_{\rho,\varepsilon}\cap G^{-1}(\clo B^p_\eta\setminus\{0\})\subset N_\varepsilon\setminus V_G$; every point of this set lies outside $M_\rho(G)$, whence the submersivity. The last claim is precisely Definition \ref{def:rhoreg}.
\end{proof}

\begin{remark}
A converse holds under a uniform formulation of the quantifiers; see \cite[Lemma 6.10]{RST24} for the Euclidean case. We will not need it: for the obstruction results below we use the direct conic argument of Lemma \ref{lem:conic} instead.
\end{remark}

\begin{corollary}\label{cor:tube}
If $G$ has isolated critical value and satisfies condition (b) associated with $\rho$, then $G$ admits a Milnor fibration on the $\rho$-tube.
\end{corollary}

The next lemma shows that, for homogeneous data, condition (b) is not merely sufficient but also necessary for the tube fibration.

\begin{lemma}[Conic obstruction]\label{lem:conic}
Suppose that $G_1,\dots,G_p$ are homogeneous polynomials (possibly of different degrees) and that $\rho$ is a homogeneous polynomial control function, positive off the origin. If the Milnor condition (b) associated with $\rho$ fails, then for every $\varepsilon>0$ and every $\eta>0$ the restriction
\[
G\big|:\ S^{m-1}_{\rho,\varepsilon}\cap G^{-1}\bigl(\clo B^p_\eta\setminus\{0\}\bigr)\longrightarrow \clo B^p_\eta\setminus\{0\}
\]
has critical points. Consequently, $G$ admits no Milnor fibration on the $\rho$-tube.
\end{lemma}

\begin{proof}
All defining equations of $M_\rho(G)$ (vanishing of the $(p+1)$-minors of the extended Jacobian) and of $V_G$ are homogeneous, so both sets are cones. By the failure of \eqref{eq:condb} there are $x_0\in V_G\setminus\{0\}$ and $p_n\in M_\rho(G)\setminus V_G$ with $p_n\to x_0$. Let $d=\deg\rho$ and put $r_n:=(\varepsilon/\rho(p_n))^{1/d}$ and $q_n:=r_np_n$. Then $\rho(q_n)=\varepsilon$, i.e.\ $q_n\in S^{m-1}_{\rho,\varepsilon}$; since $M_\rho(G)$ and $V_G$ are cones, $q_n\in M_\rho(G)\setminus V_G$; and $r_n\to r_0:=(\varepsilon/\rho(x_0))^{1/d}\in(0,\infty)$, so $q_n\to r_0x_0\in V_G\cap S^{m-1}_{\rho,\varepsilon}$. Hence $G(q_n)\to G(r_0x_0)=0$ while $G(q_n)\ne0$; for every $\eta>0$, infinitely many $q_n$ satisfy $0<\norm{G(q_n)}<\eta$. Since $\nabla\rho(q_n)\ne0$ and $q_n\in M_\rho(G)$, the restriction $G|_{S^{m-1}_{\rho,\varepsilon}}$ is critical at $q_n$.

Now suppose, for contradiction, that the fibration \eqref{eq:tubefib} exists for some $\varepsilon,\eta$, and pick $q=q_n$ as above with $y:=G(q)$, $0<\norm y<\eta$. Choose an open neighborhood $U\ni y$ inside the open punctured ball $B^p_\eta\setminus\{0\}$ and a trivialization $\varphi:G|^{-1}(U)\xrightarrow{\ \cong\ }U\times F$ over $U$, a diffeomorphism of manifolds with boundary commuting with the projections. Since $U$ avoids $S^{p-1}_\eta$, the boundary of $G|^{-1}(U)$ is exactly $S^{m-1}_{\rho,\varepsilon}\cap G^{-1}(U)$, which $\varphi$ carries onto $U\times\partial F$. Hence $G$ restricted to this boundary stratum is, up to diffeomorphism, the projection $U\times\partial F\to U$: a submersion at every point. But $q$ lies in this stratum and is a critical point of $G|_{S^{m-1}_{\rho,\varepsilon}}$  a contradiction.
\end{proof}

\subsection{An ICIS-type sufficient criterion}

\begin{proposition}\label{prop:icis}
Let $G:(\R^m,0)\to(\R^p,0)$ with $\dim V_G>0$ and $\Sing G\cap V_G=\{0\}$. Then $G$ satisfies the Milnor condition (b) associated with \emph{every} admissible real analytic control function $\rho$.
\end{proposition}

\begin{proof}
Suppose condition \eqref{eq:condb} fails. Failure as a germ means that in every neighborhood of the origin there is a point $x_0\in\clo{M_\rho(G)\setminus V_G}\cap V_G$ with $x_0\ne0$; fix such an $x_0$ inside a ball where $\Sing G\cap V_G=\{0\}$ holds, and a sequence $p_n\in M_\rho(G)\setminus V_G$ with $p_n\to x_0$.

Since $x_0\in V_G\setminus\{0\}$, we have $x_0\notin\Sing G$; as $\Sing G$ is closed, $dG$ has rank $p$ on a neighborhood $W$ of $x_0$, and the Gram matrix $\Gamma(x)=\bigl(\langle\nabla G_i(x),\nabla G_j(x)\rangle\bigr)_{i,j}$ is invertible on $W$ with continuous inverse. For large $n$, $p_n\in W\setminus\Sing G$, so membership in $M_\rho(G)$ forces
\[
\nabla\rho(p_n)=\sum_{i=1}^p a_i(p_n)\nabla G_i(p_n),\qquad a(p_n)=\Gamma(p_n)^{-1}\beta(p_n),\ \ \beta_i=\langle\nabla\rho,\nabla G_i\rangle.
\]
Letting $n\to\infty$ gives $\nabla\rho(x_0)\in\spann\{\nabla G_1(x_0),\dots,\nabla G_p(x_0)\}=\bigl(T_{x_0}V_G\bigr)^\perp$, where $V_G$ is smooth at $x_0$ with $T_{x_0}V_G=\ker dG(x_0)$ because $x_0$ is a regular point of $G$ lying on $G^{-1}(0)$. Hence $x_0$ is a critical point of $\rho|_{V_G\setminus\{0\}}$.

At points of $V_G\setminus\{0\}$ near the origin (where $dG$ has rank $p$), criticality of $\rho|_{V_G}$ is equivalent to membership in $M_\rho(G)$; thus the critical set of $\rho|_{V_G\setminus\{0\}}$ equals $(M_\rho(G)\cap V_G)\setminus\{0\}$ there, an analytic set. Since failure of \eqref{eq:condb} produces such critical points arbitrarily close to $0$, the Curve Selection Lemma yields an analytic curve $\alpha:(0,\delta)\to V_G\setminus\{0\}$ of critical points with $\alpha(t)\to0$. Then $\frac{d}{dt}(\rho\circ\alpha)(t)=\langle\nabla\rho(\alpha(t)),\alpha'(t)\rangle=0$ because $\alpha'(t)\in T_{\alpha(t)}V_G$, so $\rho\circ\alpha$ is constant; the constant is $\lim_{t\to0}\rho(\alpha(t))=\rho(0)=0$, forcing $\alpha\equiv0$ since $\rho^{-1}(0)=\{0\}$---a contradiction.
\end{proof}

Proposition \ref{prop:icis} shows that the choice of $\rho$ only matters for germs with $\Sing G\cap V_G\ne\{0\}$; all examples below are of this kind.

\subsection{Examples}

\begin{example}\label{ex:basic}
Let $G(x,y,z)=(xz,yz)$ and $\rho(x,y,z)=x^2+y^2+z^4$; note $\nabla\rho=(2x,2y,4z^3)$ vanishes only at $0$, so $\rho$ defines the origin. Here $\Sing G=\{z=0\}$, $V_G=\{z=0\}\cup\{x=y=0\}$, so $\Disc(G)=\{0\}$, and
\[
\det\begin{pmatrix}z&0&x\\0&z&y\\2x&2y&4z^3\end{pmatrix}=2z(2z^4-x^2-y^2),
\qquad
M_\rho(G)=\{z=0\}\cup\{x^2+y^2=2z^4\}.
\]
If $(x_n,y_n,z_n)\in M_\rho(G)\setminus V_G$ converges to $(x_0,y_0,z_0)\in V_G$, then $x_0^2+y_0^2=2z_0^4$; testing against the two components of $V_G$ forces $(x_0,y_0,z_0)=0$ in both cases. Hence condition (b) holds and, by Corollary \ref{cor:tube}, $G$ admits a Milnor fibration on the $\rho$-tube.
\end{example}

\begin{example}\label{ex:nine}
Let $G(x,y,z)=\bigl(xz,\,yz(9x^2+y^2+z^2)\bigr)$, so $V_G=\{z=0\}\cup\{x=y=0\}$ and $\Sing G=\{z=0\}$.

\emph{Failure for $\rho_E=x^2+y^2+z^2$.} A direct computation gives
\[
\det(\nabla G_1,\nabla G_2,\nabla\rho_E)=2z\bigl[z^4+8x^2z^2-(3x^2-y^2)^2\bigr].
\]
For $x_0\ne0$ and $y^2\ne3x_0^2$, the positive root of the biquadratic,
\[
z(y)^2=-4x_0^2+\sqrt{16x_0^4+(3x_0^2-y^2)^2}\;>\;0,
\]
produces points $\bigl(x_0,y,z(y)\bigr)\in M_{\rho_E}(G)\setminus V_G$ which converge, as $y\to\sqrt3\,x_0$, to $(x_0,\sqrt3x_0,0)\in V_G\setminus\{0\}$. Hence $\clo{M_{\rho_E}(G)\setminus V_G}\cap V_G$ contains the punctured lines $y=\pm\sqrt3\,x$ in $\{z=0\}$, and condition (b) fails for $\rho_E$. Since $G_1,G_2,\rho_E$ are homogeneous, Lemma \ref{lem:conic} shows that $G$ admits \emph{no} Milnor fibration on the Euclidean tube.

\emph{Regularity for $\rho_2=9x^2+y^2+z^2$.} Now
\[
\det(\nabla G_1,\nabla G_2,\nabla\rho_2)=2z\bigl[z^4-(9x^2+y^2)^2\bigr]=2z(z^2-9x^2-y^2)(z^2+9x^2+y^2),
\]
so $M_{\rho_2}(G)=\{z=0\}\cup\{z^2=9x^2+y^2\}$, and any limit point of the second component on $V_G$ satisfies $z_0^2=9x_0^2+y_0^2$ together with $z_0=0$ or $x_0=y_0=0$, hence equals the origin. Condition (b) holds for $\rho_2$, and $G$ admits a Milnor fibration on the $\rho_2$-tube by Corollary \ref{cor:tube}.
\end{example}

\section{The family \texorpdfstring{$G_\rho$}{Grho} and the breaking criterion}\label{sec:family}

Let $\rho(x,y,z)=ax^{2m}+by^{2n}+cz^{2k}$ with $a,b,c>0$ and integers $m,n,k\ge1$, and consider
\[
G_\rho(x,y,z)=\bigl(xz,\;yz\,\rho(x,y,z)\bigr).
\]
Then $V_{G_\rho}=\{z=0\}\cup\{x=y=0\}$ for every choice of parameters.

\begin{proposition}\label{prop:family}
Each $G_\rho$ has isolated critical value and satisfies the Milnor condition (b) associated with its own control function $\rho$; consequently it admits a Milnor fibration on the $\rho$-tube.
\end{proposition}

\begin{proof}
Writing $G_1=xz$, $G_2=yz\rho$, we have $\nabla G_2=yz\nabla\rho+(0,z\rho,y\rho)$, and row reduction gives
\begin{equation}\label{eq:selfdet}
\begin{aligned}
\det(\nabla G_1,\nabla G_2,\nabla\rho)
&=\det\begin{pmatrix}z&0&x\\0&z\rho&y\rho\\\partial_x\rho&\partial_y\rho&\partial_z\rho\end{pmatrix}
=\rho\, z\bigl(z\partial_z\rho-y\partial_y\rho-x\partial_x\rho\bigr)\\
&=2\rho\,z\bigl(ckz^{2k}-bny^{2n}-amx^{2m}\bigr).
\end{aligned}
\end{equation}
Since $\rho>0$ off the origin, $M_\rho(G_\rho)=\{z=0\}\cup\{ckz^{2k}=amx^{2m}+bny^{2n}\}$. A sequence in the second component with $z\ne0$ converging to a point of $V_{G_\rho}$ satisfies, in the limit, $ckz_0^{2k}=amx_0^{2m}+bny_0^{2n}$ together with either $z_0=0$ or $x_0=y_0=0$; positivity of the coefficients forces $(x_0,y_0,z_0)=0$ in both cases, proving (b).

For the isolated critical value, the $2\times2$ minor of $dG_\rho$ formed by the first two columns is $z^2(\rho+2bny^{2n})$, which vanishes only on $\{z=0\}$ (a vanishing sum of two nonnegative terms with $z\ne0$ would force $\rho=0$, i.e.\ the origin). On $\{z=0\}$ the first two columns of $dG_\rho$ vanish, so $\Sing G_\rho=\{z=0\}\subset V_{G_\rho}$ and $\Disc(G_\rho)=\{0\}$. Corollary \ref{cor:tube} concludes.
\end{proof}

We now fix $m=n=k=1$ and test a member of the family against a \emph{different} elliptic control function. This is the comparison criterion announced in the introduction.

\begin{theorem}[Breaking criterion]\label{thm:criterion}
Let $\rho_1=a_1x^2+b_1y^2+c_1z^2$ and $\rho_2=a_2x^2+b_2y^2+c_2z^2$ with all coefficients positive, let $G_{\rho_1}=(xz,\,yz\rho_1)$, and set
\[
\lambda:=\frac{a_1b_2}{a_2b_1}.
\]
\begin{enumerate}
\item[(i)] If $0<\lambda<9$, or if $\lambda=9$ and $a_1c_2<5a_2c_1$, then $G_{\rho_1}$ satisfies the Milnor condition (b) associated with $\rho_2$; in particular it is $\rho_2$-regular and admits a Milnor fibration on the $\rho_2$-tube.
\item[(ii)] If $\lambda>9$, or if $\lambda=9$ and $a_1c_2\ge5a_2c_1$, then condition (b) associated with $\rho_2$ fails and $G_{\rho_1}$ admits no Milnor fibration on the $\rho_2$-tube.
\end{enumerate}
\end{theorem}

\begin{proof}
\emph{Step 1: the Milnor set.} With
\[
\nabla G_2=yz\,\nabla\rho_1+(0,\,z\rho_1,\,y\rho_1),
\qquad
\nabla\rho_2=(2a_2x,\,2b_2y,\,2c_2z),
\]
expansion of the extended Jacobian determinant along its first row and collection of terms yields the identity
\begin{equation}\label{eq:Qident}
\det(\nabla G_1,\nabla G_2,\nabla\rho_2)=-2z\,Q(x^2,y^2,z^2),
\end{equation}
where, in the variables $X=x^2$, $Y=y^2$, $Z=z^2$,
\[
\begin{split}
Q(X,Y,Z)=a_1a_2X^2+(3a_2b_1-a_1b_2)XY+b_1b_2Y^2\qquad\qquad\\
+(a_2c_1-a_1c_2)XZ+3(b_2c_1-b_1c_2)YZ-c_1c_2Z^2.
\end{split}
\]
Hence $M_{\rho_2}(G_{\rho_1})=\{z=0\}\cup\{Q(x^2,y^2,z^2)=0\}$, and condition (b) concerns the accumulation of $\{Q=0\}\setminus V_{G_{\rho_1}}$ on $V_{G_{\rho_1}}\setminus\{0\}$.

\emph{Step 2: no accumulation on the $z$-axis.} For $p_0=(0,0,z_0)$ with $z_0\ne0$ we have $Q(0,0,z_0^2)=-c_1c_2z_0^4\ne0$; by continuity of $Q$, no sequence in $\{Q=0\}$ converges to $p_0$. Any accumulation on $V_{G_{\rho_1}}\setminus\{0\}$ must therefore occur on $\{z=0\}\setminus\{0\}$, i.e.\ at points $(x_0,y_0,0)$ with $(x_0,y_0)\ne(0,0)$, along sequences with $Z\to0$.

\emph{Step 3: no accumulation with $x_0=0$.} Suppose $(x_n,y_n,z_n)\in\{Q=0\}$, $z_n\ne0$, converges to $(0,y_0,0)$ with $y_0\ne0$. Dividing $Q(X_n,Y_n,Z_n)=0$ by $Y_n^2$ and writing $N_n=X_n/Y_n\to0$, $S_n=Z_n/Y_n\to0$, we get in the limit $b_1b_2=0$, a contradiction. Hence any accumulation point has $x_0\ne0$; along such sequences the ratios $M_n:=Y_n/X_n\to y_0^2/x_0^2=:M_*\in[0,\infty)$ and $T_n:=Z_n/X_n\to0$ are well defined and finite.

\emph{Step 4: reduction.} Dividing $Q=0$ by $X^2$ and writing $M=Y/X\ge0$, $T=Z/X>0$, we obtain $q(M,T)=0$ with
\begin{equation}\label{eq:qMT}
q(M,T)=b_1b_2M^2+(3a_2b_1-a_1b_2)M+a_1a_2
+T\bigl[(a_2c_1-a_1c_2)+3(b_2c_1-b_1c_2)M\bigr]-c_1c_2T^2,
\end{equation}
a polynomial of degree exactly $2$ in $M$ (with leading coefficient $b_1b_2>0$) and degree $2$ in $T$. Accumulation at a point of $\{z=0\}$ with $x_0\ne0$ therefore forces, in the limit, $q(M_*,0)=0$ for some $M_*\ge0$:
\begin{equation}\label{eq:binary}
b_1b_2M_*^2+(3a_2b_1-a_1b_2)M_*+a_1a_2=0,\qquad M_*\ge0.
\end{equation}

\emph{Step 5: solvability of \eqref{eq:binary}.} The discriminant of \eqref{eq:binary} is
\[
(3a_2b_1-a_1b_2)^2-4a_1a_2b_1b_2=(a_2b_1)^2\bigl[(3-\lambda)^2-4\lambda\bigr]=(a_2b_1)^2(\lambda-1)(\lambda-9).
\]
The product of the roots is $a_1a_2/(b_1b_2)>0$ (so real roots share their sign) and their sum is $a_2(\lambda-3)/b_2$. If $0<\lambda\le1$, the roots are real but both negative; if $1<\lambda<9$, they are not real; note also that $M_*=0$ is never a root, since $q(0,0)=a_1a_2>0$. Hence \eqref{eq:binary} has a nonnegative solution if and only if $\lambda\ge9$, in which case the roots are positive (their sum $a_2(\lambda-3)/b_2>0$). For $\lambda=9$ the root is double, equal to
\[
M_0=-\frac{3a_2b_1-a_1b_2}{2b_1b_2}=\frac{3a_2}{b_2}>0,
\]
using $a_1b_2=9a_2b_1$.

\emph{Step 6: the case $0<\lambda<9$.} By Step 5 there is no $M_*\ge0$ solving \eqref{eq:binary}; combined with Steps 2--4, no sequence of $\{Q=0\}\setminus V_{G_{\rho_1}}$ accumulates on $V_{G_{\rho_1}}\setminus\{0\}$. Condition (b) holds, and Corollary \ref{cor:tube} yields the $\rho_2$-tube fibration.

\emph{Step 7: the case $\lambda>9$.} Let $M_0>0$ be one of the two (simple) positive roots of \eqref{eq:binary}; then $\partial_Mq(M_0,0)\ne0$, and the Implicit Function Theorem provides an analytic function $M(T)$, defined for $|T|$ small, with $M(0)=M_0$ and $q(M(T),T)=0$; by continuity $M(T)>0$ for small $T$. Fixing $X>0$ and letting $T\to0^+$, the points
\[
\Bigl(\sqrt X,\ \sqrt{M(T)X},\ \sqrt{TX}\Bigr)\in\{Q=0\}\setminus V_{G_{\rho_1}}
\]
converge to $\bigl(\sqrt X,\sqrt{M_0X},0\bigr)\in\{z=0\}\setminus\{0\}$. Condition (b) fails.

\emph{Step 8: the borderline case $\lambda=9$.} Since $M_0$ is a double root, $q(M,0)=b_1b_2(M-M_0)^2$, and substituting $M=M_0+u$ in \eqref{eq:qMT} gives the \emph{exact} identity
\begin{equation}\label{eq:qexact}
q(M_0+u,T)=b_1b_2\,u^2+3(b_2c_1-b_1c_2)\,uT+\kappa\,T-c_1c_2T^2,
\end{equation}
where
\[
\kappa:=(a_2c_1-a_1c_2)+3(b_2c_1-b_1c_2)\,M_0 .
\]
Substituting $M_0=3a_2/b_2$ and using $9a_2b_1/b_2=a_1$ (from $\lambda=9$), one computes
\[
\kappa=(a_2c_1-a_1c_2)+9a_2c_1-\tfrac{9a_2b_1}{b_2}c_2=10a_2c_1-2a_1c_2=2(5a_2c_1-a_1c_2).
\]
Since \eqref{eq:qexact} is an honest quadratic in $u$, the existence of real solutions $M=M_0+u$ of $q(M,T)=0$, for fixed $T>0$, is decided by its discriminant
\[
D(T)=9(b_2c_1-b_1c_2)^2T^2-4b_1b_2\bigl(\kappa T-c_1c_2T^2\bigr)
=-4b_1b_2\kappa\,T+\bigl[9(b_2c_1-b_1c_2)^2+4b_1b_2c_1c_2\bigr]T^2,
\]
whose $T^2$-coefficient is strictly positive.

If $\kappa>0$ (i.e.\ $a_1c_2<5a_2c_1$), then $D(T)<0$ for $0<T<T_1:=\dfrac{4b_1b_2\kappa}{9(b_2c_1-b_1c_2)^2+4b_1b_2c_1c_2}$; since $q(\cdot,T)$ is quadratic in $M$ with discriminant $D(T)$, it has \emph{no real roots at all} for such $T$. Along any putative accumulating sequence we have $T_n\to0$ (Step 3), so eventually $T_n<T_1$ and $q(M_n,T_n)=0$ is impossible. Together with Steps 2--3, condition (b) holds; Corollary \ref{cor:tube} gives the $\rho_2$-tube fibration.

If $\kappa\le0$ (i.e.\ $a_1c_2\ge5a_2c_1$), then $D(T)>0$ for all small $T>0$: for $\kappa<0$ both summands of $D(T)$ are positive, while for $\kappa=0$ one has $D(T)=\bigl[9(b_2c_1-b_1c_2)^2+4b_1b_2c_1c_2\bigr]T^2>0$. The quadratic formula then provides real roots
\[
u_\pm(T)=\frac{-3(b_2c_1-b_1c_2)T\pm\sqrt{D(T)}}{2b_1b_2}=O\bigl(\sqrt T\,\bigr)\xrightarrow[T\to0^+]{}0,
\]
so $M(T):=M_0+u_+(T)>0$ for small $T>0$ and $q(M(T),T)=0$. As in Step 7, this produces points of $\{Q=0\}\setminus V_{G_{\rho_1}}$ accumulating at $\bigl(\sqrt X,\sqrt{M_0X},0\bigr)\ne0$, and condition (b) fails.

\emph{Step 9: non-existence of the tube fibration in the failure cases.} All the data $G_1$ (degree $2$), $G_2$ (degree $4$) and $\rho_2$ (degree $2$) are homogeneous, so Lemma \ref{lem:conic} applies whenever condition (b) fails, excluding the Milnor fibration on the $\rho_2$-tube. This completes the proof.
\end{proof}

\begin{remark}\label{rem:threshold}
Example \ref{ex:nine} is the borderline case of Theorem \ref{thm:criterion}: there $\rho_1=9x^2+y^2+z^2$ ($a_1=9$, $b_1=c_1=1$) and $\rho_2=\rho_E$ ($a_2=b_2=c_2=1$), so $\lambda=9$ with $a_1c_2=9\ge5=5a_2c_1$, and the criterion predicts the failure of condition (b) for $\rho_E$---in agreement with the direct computation, which located the accumulation exactly along $y=\pm\sqrt3\,x$, i.e.\ at the double root $M_0=3$ of \eqref{eq:binary}.
\end{remark}

\begin{example}\label{ex:window}
Let $G=(xz,\,yz(x^2+y^2+z^2))$, i.e.\ $\rho_1=\rho_E$, tested against $\rho_2=x^2+4y^2+z^2$. Then $\lambda=4\in(0,9)$ and Theorem \ref{thm:criterion}(i) gives $\rho_2$-regularity. Directly: the binary form \eqref{eq:binary} becomes $4M_*^2-M_*+1=0$, with discriminant $1-16=-15<0$, confirming the absence of real accumulation directions.
\end{example}

\section{Quasi-homogeneous germs}\label{sec:qh}

\subsection{Resonant Euclidean incompatibility}

For germs whose second component involves higher even powers, the obstruction to Euclidean regularity is governed by a resonance between the coefficients and the exponents, as announced in the introduction.

\begin{theorem}\label{thm:resonant}
Let $H(x,y,z)=ax^{2s}+by^{2s}+z^2$ with $a,b>0$ and an integer $s\ge2$, let $G=(xz,\,yzH)$, and let $\rho_E=x^2+y^2+z^2$. If
\[
\frac ab=\frac{s+1}{s-1},
\]
then the Milnor condition (b) associated with $\rho_E$ fails: for every $u\ne0$ small, the points $(u,\pm u,0)$ belong to $\clo{M_{\rho_E}(G)\setminus V_G}\cap V_G$.
\end{theorem}

\begin{proof}
Under the reflection $S(x,y,z)=(x,y,-z)$ one has $G\circ S=-G$ and $\rho_E\circ S=\rho_E$, whence $(\nabla G_i)\circ S=-S\nabla G_i$ and $(\nabla\rho_E)\circ S=S\nabla\rho_E$; since $\det S=-1$, the extended Jacobian determinant is odd in $z$. Writing it as $2z\,Q(x,y,z)$, the factor $Q$ is therefore \emph{even} in $z$, so that $Q(x,y,z)=Q(x,y,0)+O(z^2)$ with all correction terms carrying a factor $z^2$. A direct computation gives $Q(x,y,0)=-F(x,y)$ with
\[
F(x,y)=ax^{2s+2}-(2s-1)a\,x^{2s}y^2+(2s+1)b\,x^2y^{2s}+b\,y^{2s+2}.
\]
Along $y=x$,
\[
F(x,x)=\bigl[a-(2s-1)a+(2s+1)b+b\bigr]x^{2s+2}=2\bigl[(s+1)b-(s-1)a\bigr]x^{2s+2},
\]
which vanishes identically precisely under the resonance $a/b=(s+1)/(s-1)$; since $F$ is even in $y$, also $F(x,-x)\equiv0$.

Next,
\[
\partial_yF(u,u)=2u^{2s+1}\bigl[-(2s-1)a+(2s^2+2s+1)b\bigr]
=\frac{4bs(s^2-s-1)}{s-1}\,u^{2s+1},
\]
after substituting $a=(s+1)b/(s-1)$. As $s\ge2$ is an integer and the roots of $s^2-s-1$ are $(1\pm\sqrt5)/2$, we get $\partial_yF(u,u)\ne0$ for all $u\ne0$. Since $Q$ is even in $z$, $\partial_yQ(u,u,0)=-\partial_yF(u,u)\ne0$, and the Implicit Function Theorem produces an analytic curve $\varphi_u(z)$ with $\varphi_u(0)=u$ and $Q(u,\varphi_u(z),z)=0$ for small $z$. For $z\ne0$, the point $q_z=(u,\varphi_u(z),z)$ lies in $M_{\rho_E}(G)$ (as $\det=2zQ$ vanishes) and outside $V_G$ (as $G_1(q_z)=uz\ne0$), and $q_z\to(u,u,0)\in V_G\setminus\{0\}$ as $z\to0$. The case $y=-x$ is identical by evenness. Since $u$ is arbitrarily small, condition \eqref{eq:condb} fails.
\end{proof}

\begin{example}\label{ex:superell}
The germ $G=(xz,\,yz(3x^4+y^4+z^2))$ satisfies the resonance of Theorem \ref{thm:resonant} with $s=2$, $a=3$, $b=1$ (indeed $3=(2+1)/(2-1)$), so condition (b) fails for $\rho_E$. Here the limit curve factors completely,
\[
3x^6-9x^4y^2+5x^2y^4+y^6=(x-y)(x+y)\bigl(3x^4-6x^2y^2-y^4\bigr),
\]
so its real zero set is the union of the four lines $y^2=x^2$ and $y^2=(2\sqrt3-3)\,x^2$: in the variable $t=y^2/x^2$ the nonnegative roots are $t=1$ and $t=2\sqrt3-3\approx0.4641$, both simple, so the Implicit Function Theorem step in the proof of Theorem \ref{thm:resonant} applies along each of them. By contrast, for the adapted superellipsoidal control function $\rho_2=3x^4+y^4+z^2$, formula \eqref{eq:selfdet} gives $M_{\rho_2}(G)=\{z=0\}\cup\{z^2=6x^4+2y^4\}$, whose closure meets $V_G$ only at the origin; Proposition \ref{prop:family} (with $m=n=2$, $k=1$) then provides the Milnor fibration on the $\rho_2$-tube.
\end{example}

The superellipsoidal control is not, however, the only possible remedy. The resonance of Theorem \ref{thm:resonant} turns out to be a statement about the Euclidean \emph{normalisation} of the quadratic control, and the whole diagonal quadratic family can be analysed at once.

\begin{proposition}\label{prop:quadfix}
Let $H=ax^{2s}+by^{2s}+z^2$ with $a,b>0$ and $s\ge2$ an integer, let $G=(xz,\,yzH)$, and let
\[
\rho_{\alpha,\beta,\gamma}=\alpha x^2+\beta y^2+\gamma z^2,\qquad \alpha,\beta,\gamma>0 .
\]
Then $\det(\nabla G_1,\nabla G_2,\nabla\rho_{\alpha,\beta,\gamma})=2z\,Q$ where, in the variable $Z=z^2$,
\begin{equation}\label{eq:quadfix}
Q=-B(x,y)+L(x,y)\,Z+\gamma Z^2
\end{equation}
with
\begin{align}
B(x,y)&=a\alpha\,x^{2s+2}-(2s-1)a\beta\,x^{2s}y^2+(2s+1)b\alpha\,x^2y^{2s}+b\beta\,y^{2s+2},\label{eq:Bform}\\
L(x,y)&=\gamma a\,x^{2s}+(2s+1)\gamma b\,y^{2s}-\alpha x^2-3\beta y^2 .\label{eq:Lform}
\end{align}
Consequently $G$ satisfies the Milnor condition (b) associated with $\rho_{\alpha,\beta,\gamma}$ if and only if
\[
B(x,y)\ \ge\ 0\qquad\text{for all }(x,y).
\]
In particular the condition does not involve $\gamma$ and depends on $(\alpha,\beta)$ only through the ratio $\mu:=\alpha/\beta$. Moreover $B$ vanishes along $y=\pm x$ exactly when
\begin{equation}\label{eq:mustar}
\mu=\frac{(2s-1)a-b}{a+(2s+1)b},
\end{equation}
and under the resonance $a/b=(s+1)/(s-1)$ of Theorem \ref{thm:resonant} the right-hand side of \eqref{eq:mustar} equals $1$.
\end{proposition}

\begin{proof}
Write
\[
\nabla G_1=(z,\,0,\,x),
\qquad
\nabla G_2=(yzH_x,\ zH+yzH_y,\ yH+yzH_z).
\]
Expanding the extended Jacobian determinant along its first row and factoring out $2z$ gives $\det(\nabla G_1,\nabla G_2,\nabla\rho_{\alpha,\beta,\gamma})=2zQ$ with
\[
Q=\gamma z^2H+\gamma yz^2H_y-\beta y^2H-\beta y^2zH_z+\beta xy^2H_x-\alpha x^2H-\alpha x^2yH_y .
\]
Substituting $H=ax^{2s}+by^{2s}+Z$, $H_x=2sax^{2s-1}$, $H_y=2sby^{2s-1}$, $H_z=2z$ and collecting powers of $Z=z^2$ yields \eqref{eq:quadfix} with $B$ and $L$ as in \eqref{eq:Bform}--\eqref{eq:Lform}. In particular $B=-Q|_{z=0}$ and $Q(0,0,z)=\gamma z^4$.

Note next that $V_G=\{z=0\}\cup\{x=y=0\}$: if $z\ne0$ then $G_1=0$ forces $x=0$, and then $H=by^{2s}+z^2>0$, so $G_2=0$ forces $y=0$. Since $M_{\rho_{\alpha,\beta,\gamma}}(G)=\{z=0\}\cup\{Q=0\}$ and $\{z=0\}\subset V_G$, condition (b) asks precisely that no sequence of points of $\{Q=0\}$ with $z\ne0$ and $(x,y)\ne(0,0)$ converge to a point of $V_G\setminus\{0\}$. On the $z$-axis $Q=\gamma z^4\ne0$ for $z\ne0$, so such a limit point must have the form $(x_0,y_0,0)$ with $(x_0,y_0)\ne(0,0)$, and then necessarily $B(x_0,y_0)=0$. This necessary condition is \emph{not} sufficient: whether the zero of $B$ is actually reached is decided by the transverse term $L$.

Since $2s\ge4$, the monomials $x^{2s}$ and $y^{2s}$ in \eqref{eq:Lform} are $o(x^2)$ and $o(y^2)$ as $(x,y)\to0$. Hence there are $\delta>0$ and $c>0$ with
\begin{equation}\label{eq:Lbound}
L(x,y)\le-c\,(x^2+y^2)\qquad\text{for }0<\norm{(x,y)}<\delta ;
\end{equation}
one may take $c=\tfrac12\min\{\alpha,3\beta\}$.

\emph{Sufficiency.} Assume $B\ge0$ everywhere, and let $(x,y,z)$ satisfy $Q=0$ with $Z=z^2>0$ and $0<\norm{(x,y)}<\delta$. By \eqref{eq:quadfix} and \eqref{eq:Lbound},
\[
\gamma Z^2=B(x,y)-L(x,y)\,Z\ \ge\ -L(x,y)\,Z\ \ge\ c\,(x^2+y^2)\,Z ,
\]
whence $Z\ge(c/\gamma)(x^2+y^2)$. A sequence of such points converging to $(x_0,y_0,0)$ with $(x_0,y_0)\ne(0,0)$ would have $Z\to0$ while $(c/\gamma)(x^2+y^2)\to(c/\gamma)(x_0^2+y_0^2)>0$, which is impossible. Condition (b) holds.

\emph{Necessity.} Assume $B(x_1,y_1)<0$ for some $(x_1,y_1)$. As $B$ is homogeneous of degree $2s+2$, this is a property of the direction, so we may take $\norm{(x_1,y_1)}=r$ for any prescribed $0<r<\delta$. On the coordinate axes $B(x,0)=a\alpha x^{2s+2}>0$ and $B(0,y)=b\beta y^{2s+2}>0$, so on the circle of radius $r$ the continuous function $B$ changes sign; choose a point $(x_0,y_0)$ of that circle with $B(x_0,y_0)=0$ which is a limit of points $(x_n,y_n)$ of the circle with $B(x_n,y_n)<0$. Put $B_n:=B(x_n,y_n)<0$ and $L_n:=L(x_n,y_n)\le-cr^2<0$. The quadratic $\gamma Z^2+L_nZ-B_n$ has product of roots $-B_n/\gamma>0$ and sum $-L_n/\gamma>0$, and its discriminant $L_n^2+4\gamma B_n$ is positive once $|B_n|$ is small enough; it therefore has two positive roots, the smaller being
\[
Z_n=\frac{-L_n-\sqrt{L_n^2+4\gamma B_n}}{2\gamma}=\frac{B_n}{L_n}+O(B_n^2)\ \xrightarrow[\ n\to\infty\ ]{}\ 0 .
\]
The points $\bigl(x_n,y_n,\sqrt{Z_n}\bigr)$ lie in $\{Q=0\}\setminus V_G$ and converge to $(x_0,y_0,0)\in V_G\setminus\{0\}$, so condition (b) fails.

The independence of $\gamma$ and the homogeneity in $(\alpha,\beta)$ are immediate from \eqref{eq:Bform}. Finally, putting $y=\pm x$ in $B$ gives $x^{2s+2}\bigl[a\alpha-(2s-1)a\beta+(2s+1)b\alpha+b\beta\bigr]$, which vanishes precisely under \eqref{eq:mustar}. Substituting $a=s+1$ and $b=s-1$ there,
\[
\frac{(2s-1)(s+1)-(s-1)}{(s+1)+(2s+1)(s-1)}=\frac{2s^2}{2s^2}=1 . \qedhere
\]
\end{proof}

\begin{remark}\label{rem:threshold42}
For the germ of Example \ref{ex:superell} ($s=2$, $a=3$, $b=1$) Proposition \ref{prop:quadfix} is completely explicit. Normalising $\beta=1$ and writing $\mu=\alpha$, the form \eqref{eq:Bform} becomes $B(x,y)=x^6p(y^2/x^2)$ with
\[
p(t)=t^3+5\mu t^2-9t+3\mu ,
\]
and since $p(t)\ge0$ is equivalent to $\mu\,(5t^2+3)\ge9t-t^3$, condition (b) associated with $\rho_{\mu,1,\gamma}$ holds if and only if
\[
\mu\ \ge\ \mu_0:=\max_{t\ge0}\frac{t\,(9-t^2)}{5t^2+3}=\frac{3}{25}\sqrt{40\sqrt6-15}=1.0931176\ldots,
\]
the positive root of $125u^4+54u^2-243=0$, the maximum being attained at $t_0^2=(12\sqrt6-27)/5$. The limiting value $\mu=\mu_0$ is \emph{included}: there $p$ has a double zero at $t_0$ and remains nonnegative, so $B$ vanishes along the pair of lines $y^2=t_0x^2$ without ever becoming negative, and by Proposition \ref{prop:quadfix} condition (b) still holds --- no branch of $\{Q=0\}$ with $z\ne0$ reaches those lines. The Euclidean control is the case $\mu=1$, which lies strictly below the threshold: there $p(t)=(t-1)(t^2+6t-3)$ does take negative values, in agreement with Theorem \ref{thm:resonant}. By contrast $\tilde\rho=4x^2+y^2+z^2$ has $\mu=4>\mu_0$; explicitly
\[
B(x,y)=12x^6-9x^4y^2+20x^2y^4+y^6=x^6\bigl(t^3+20t^2-9t+12\bigr),\qquad Q(0,0,z)=z^4,
\]
and $t^3+20t^2-9t+12>0$ for $t\ge0$ (for $0\le t\le9/20$ already $12-9t>0$, and for $t\ge9/20$ one has $20t^2-9t\ge0$), so condition (b) holds for $\tilde\rho$. Moreover the $2\times2$ minors of $dG$ are $z^2(3x^4+5y^4+z^2)$, $-yz(9x^4-y^4-3z^2)$ and $-xz(3x^4+5y^4+z^2)$, so $\Sing G=\{z=0\}$ and $\Disc(G)=\{0\}$, and Corollary \ref{cor:tube} shows that $G$ already admits a Milnor fibration on the $\tilde\rho$-tube. Thus the Euclidean obstruction in Example \ref{ex:superell} is not intrinsic to the exponents: a quadratic reweighting of the metric removes it as soon as the anisotropy $\mu$ reaches the threshold $\mu_0$. What the superellipsoidal $\rho_2$ contributes is that it is adapted to the weights of $G$ and needs no threshold at all.
\end{remark}

\subsection{Confinement to the singular locus}
As Example \ref{ex:nine} shows, quasi-homogeneity of the pair $(G,\rho_w)$ together with an isolated critical value does \emph{not} imply condition (b): the germ there is homogeneous with isolated critical value, $\rho_E$ is homogeneous with the same weights, and (b) fails. The Euler relations do, however, confine the possible accumulation to the singular locus. This is the confinement announced in the introduction; note that in Example \ref{ex:nine} the accumulation lines $\{y=\pm\sqrt3x,\ z=0\}$ do lie inside $\Sing G=\{z=0\}$, as the theorem predicts.

\begin{theorem}\label{thm:confinement}
Let $G=(G_1,G_2):(\R^3,0)\to(\R^2,0)$ be a real analytic germ whose components are quasi-homogeneous of degrees $d_1,d_2>0$ with respect to a weight vector $w=(w_x,w_y,w_z)$ with positive entries, and let $\rho_w$ be quasi-homogeneous of degree $d_\rho>0$ with respect to $w$ and strictly positive off the origin. Then
\[
\clo{M_{\rho_w}(G)\setminus V_G}\cap V_G\ \subseteq\ \Sing G\cap V_G.
\]
\end{theorem}

\begin{proof}
Let $E=w_xx\partial_x+w_yy\partial_y+w_zz\partial_z$ be the Euler vector field, so that $\langle\nabla G_i,E\rangle=d_iG_i$ and $\langle\nabla\rho_w,E\rangle=d_\rho\rho_w$. Let $p_n\in M_{\rho_w}(G)\setminus V_G$ with $p_n\to x_0\in V_G\setminus\{0\}$. Choose multipliers $(\lambda_1^{(n)},\lambda_2^{(n)},\lambda_3^{(n)})$ with
\[
\lambda_1^{(n)}\nabla G_1(p_n)+\lambda_2^{(n)}\nabla G_2(p_n)+\lambda_3^{(n)}\nabla\rho_w(p_n)=0,
\qquad
(\lambda_1^{(n)})^2+(\lambda_2^{(n)})^2+(\lambda_3^{(n)})^2=1,
\]
and pass to a subsequence with $\lambda^{(n)}\to\lambda$, $\norm\lambda=1$. Pairing the relation with $E(p_n)$ and using the Euler identities,
\[
\lambda_1^{(n)}d_1G_1(p_n)+\lambda_2^{(n)}d_2G_2(p_n)+\lambda_3^{(n)}d_\rho\rho_w(p_n)=0.
\]
Since $G(p_n)\to0$ and $\rho_w(p_n)\to\rho_w(x_0)>0$, the limit reads $\lambda_3d_\rho\rho_w(x_0)=0$, whence $\lambda_3=0$ and $(\lambda_1,\lambda_2)\ne(0,0)$. Passing to the limit in the original relation gives $\lambda_1\nabla G_1(x_0)+\lambda_2\nabla G_2(x_0)=0$, i.e.\ $x_0\in\Sing G$.
\end{proof}

\section{The fibration on the \texorpdfstring{$\rho$}{rho}-sphere}\label{sec:sphere}

Let $G:(\R^m,0)\to(\R^p,0)$, $m>p\ge2$, and let $\rho$ be an admissible analytic control function. Off $V_G$ we define the projection
\[
\Psi_G(x)=\frac{G(x)}{\norm{G(x)}}\in S^{p-1},
\]
and for small $\varepsilon>0$ we write $K_{\rho,\varepsilon}:=S^{m-1}_{\rho,\varepsilon}\cap V_G$ for the link.

\begin{definition}\label{def:spherefib}
The germ $G$ \emph{admits a Milnor fibration on the $\rho$-sphere} if, for every sufficiently small $\varepsilon>0$, the restriction
\begin{equation}\label{eq:spherefib}
\Psi_G:\ S^{m-1}_{\rho,\varepsilon}\setminus K_{\rho,\varepsilon}\longrightarrow S^{p-1}
\end{equation}
is a smooth locally trivial fibration.
\end{definition}

\begin{definition}\label{def:angularmilnor}
The \emph{Milnor set of the projection associated with $\rho$} is
\[
M_\rho(\Psi_G):=\bigl\{x\notin V_G:\ \Psi_G\big|_{\rho^{-1}(\rho(x))}\ \text{is not a submersion at } x\bigr\}.
\]
We say $\Psi_G$ is \emph{$\rho$-regular} if $M_\rho(\Psi_G)\cap B^m_{\rho,\varepsilon_0}=\emptyset$ for some $\varepsilon_0>0$.
\end{definition}

Since any smooth locally trivial fibration is a submersion, the existence of \eqref{eq:spherefib} for all small $\varepsilon$ forces the $\rho$-regularity of $\Psi_G$; the content of Theorem \ref{thm:sphere} below is the converse, under condition (b).

When $p=2$, criticality of the projection admits a convenient vector calculus description in $\R^3$.

\begin{lemma}\label{lem:p2}
Let $m=3$, $p=2$, and identify $\Psi_G=e^{i\theta}$ with $\theta=\arg(G_1+iG_2)$, defined locally off $V_G$. Then $\nabla\theta=X/\norm G^2$ with $X:=G_1\nabla G_2-G_2\nabla G_1$, and, for $x\notin V_G$,
\[
x\in M_\rho(\Psi_G)\iff X(x)\times\nabla\rho(x)=0.
\]
\end{lemma}

\begin{proof}
Locally $d\theta=(G_1\,dG_2-G_2\,dG_1)/\norm G^2$, whence the gradient formula. The target being one-dimensional, $\Psi_G$ restricted to the level $\rho^{-1}(\rho(x))$ fails to be a submersion at $x$ exactly when $d\theta$ annihilates the tangent space of the level, i.e.\ when $\nabla\theta(x)$ is a multiple of $\nabla\rho(x)$ (including $\nabla\theta(x)=0$); in $\R^3$ this is the vanishing of the cross product.
\end{proof}

\begin{theorem}\label{thm:sphere}
Let $G:(\R^m,0)\to(\R^p,0)$, $m>p\ge2$, be real analytic with $\Disc(G)\subseteq\{0\}$, $\dim V_G>0$, and suppose $G$ satisfies the Milnor condition (b) associated with $\rho$. Then $\Psi_G$ is $\rho$-regular if and only if, for every sufficiently small $\varepsilon>0$, the restriction \eqref{eq:spherefib} is a smooth locally trivial fibration.
\end{theorem}

\begin{proof}
($\Leftarrow$) If \eqref{eq:spherefib} is a locally trivial fibration for all small $\varepsilon$, it is in particular a submersion at every point of $S^{m-1}_{\rho,\varepsilon}\setminus K_{\rho,\varepsilon}$; since every $q\in B^m_{\rho,\varepsilon_0}\setminus V_G$ lies on the $\rho$-sphere of level $\rho(q)$, we get $M_\rho(\Psi_G)\cap B^m_{\rho,\varepsilon_0}=\emptyset$.

($\Rightarrow$) Fix $\varepsilon>0$ small. By Lemma \ref{lem:btosub} there is $0<\eta\ll\varepsilon$ such that
\begin{equation}\label{eq:collarfib}
G\big|:\ S^{m-1}_{\rho,\varepsilon}\cap G^{-1}\bigl(\clo B^p_\eta\setminus\{0\}\bigr)\longrightarrow \clo B^p_\eta\setminus\{0\}
\end{equation}
is a submersion; it is proper (preimages of compacta are closed subsets of the compact $S^{m-1}_{\rho,\varepsilon}$), so it is a smooth locally trivial fibration by Ehresmann's theorem. Under the polar diffeomorphism $\clo B^p_\eta\setminus\{0\}\cong S^{p-1}\times(0,\eta]$, the composition of \eqref{eq:collarfib} with the first projection is again locally trivial: for a contractible open $U\subset S^{p-1}$, the fibration \eqref{eq:collarfib} is trivial over the contractible set $U\times(0,\eta]$, and composing with $\mathrm{pr}_1$ exhibits local triviality over $U$ with fiber $(0,\eta]\times F$. Thus
\begin{equation}\label{eq:innerpart}
\frac{G}{\norm G}\Big|:\ S^{m-1}_{\rho,\varepsilon}\cap G^{-1}\bigl(\clo B^p_\eta\setminus\{0\}\bigr)\longrightarrow S^{p-1}
\end{equation}
is a locally trivial fibration; restricting over the boundary sphere $S^{p-1}_\eta$ we obtain the same conclusion for
\begin{equation}\label{eq:boundarypart}
\frac{G}{\norm G}\Big|:\ S^{m-1}_{\rho,\varepsilon}\cap G^{-1}\bigl(S^{p-1}_\eta\bigr)\longrightarrow S^{p-1}.
\end{equation}
On the outer region, consider
\begin{equation}\label{eq:outerpart}
\Psi_G\big|:\ W:=S^{m-1}_{\rho,\varepsilon}\setminus G^{-1}\bigl(B^p_\eta\bigr)\longrightarrow S^{p-1},
\end{equation}
whose domain $W$ is a compact manifold with boundary $\partial W=S^{m-1}_{\rho,\varepsilon}\cap G^{-1}(S^{p-1}_\eta)$ and avoids $V_G$, since $V_G\subset G^{-1}(B^p_\eta)$. Ehresmann's theorem for manifolds with boundary requires two separate conditions, which we verify in turn.

(i) \emph{$\Psi_G|_W$ is a submersion.} For $x\in W$ one has $T_xW=T_xS^{m-1}_{\rho,\varepsilon}$, and $x\notin V_G$ lies on the $\rho$-sphere of level $\rho(x)=\varepsilon$. By $\rho$-regularity of $\Psi_G$ we have $x\notin M_\rho(\Psi_G)$, that is, $\Psi_G|_{S^{m-1}_{\rho,\varepsilon}}$ is a submersion at $x$.

(ii) \emph{$\Psi_G|_{\partial W}$ is a submersion.} This does \emph{not} follow from (i): the tangent space $T_x\partial W$ is a hyperplane in $T_xS^{m-1}_{\rho,\varepsilon}$, and surjectivity of $d\Psi_G$ on the larger space does not imply surjectivity on the hyperplane. It is, however, exactly what \eqref{eq:boundarypart} provides. Indeed, the domain of \eqref{eq:boundarypart} is precisely $\partial W$, and \eqref{eq:boundarypart} has just been shown to be a locally trivial fibration; in particular it is a submersion at every point of $\partial W$.

Since $W$ is compact, \eqref{eq:outerpart} is proper, and Ehresmann's theorem for manifolds with boundary makes it a locally trivial fibration. Finally, \eqref{eq:innerpart} and \eqref{eq:outerpart} are locally trivial fibrations over $S^{p-1}$ whose restrictions to their common boundary both coincide with \eqref{eq:boundarypart}; each admits a collar neighbourhood of that boundary compatible with the projection to $S^{p-1}$, so the two may be glued smoothly along collars, exactly as in the proof of \cite[Theorem 1.3]{ACT13}, yielding the locally trivial fibration \eqref{eq:spherefib}.
\end{proof}

\begin{remark}
Independence of the equivalence class of \eqref{eq:spherefib} from $\varepsilon$ can be obtained by the isotopy arguments of \cite{ACT13,ARSRT20}; we will not need it, since our non-existence result below rules out local triviality for \emph{every} radius separately.
\end{remark}

For the quadratic family, the projection is automatically regular with respect to the germ's own control function.

\begin{lemma}\label{lem:autoangular}
Let $\rho=ax^2+by^2+cz^2$ with $a,b,c>0$ and $G_\rho=(xz,\,yz\rho)$. Then $M_\rho(\Psi_{G_\rho})=\emptyset$ off $V_{G_\rho}$.
\end{lemma}

\begin{proof}
Write $F_1=xz$, $F_2=yz$, so that $G_\rho=(F_1,F_2\rho)$ and
\[
X=G_1\nabla G_2-G_2\nabla G_1=\rho\bigl(F_1\nabla F_2-F_2\nabla F_1\bigr)+F_1F_2\nabla\rho=\rho\,Y+F_1F_2\nabla\rho,
\]
with $Y:=F_1\nabla F_2-F_2\nabla F_1=xz(0,z,y)-yz(z,0,x)=z^2(-y,x,0)$. Hence $X\times\nabla\rho=\rho\,(Y\times\nabla\rho)$, and with $\nabla\rho=(2ax,2by,2cz)$ a direct computation gives
\[
Y\times\nabla\rho=2z^2\bigl(cxz,\ cyz,\ -(ax^2+by^2)\bigr).
\]
For a point off $V_{G_\rho}$ (so $z\ne0$ and $(x,y)\ne(0,0)$) and with $\rho>0$, the vanishing of $X\times\nabla\rho$ would force, from the third component, $ax^2+by^2=0$, i.e.\ $x=y=0$---a contradiction. Lemma \ref{lem:p2} concludes.
\end{proof}

\begin{corollary}\label{cor:spherefamily}
For all $a,b,c>0$, the germ $G_\rho=(xz,\,yz(ax^2+by^2+cz^2))$ admits a Milnor fibration on the $\rho$-sphere, $\rho=ax^2+by^2+cz^2$, as well as on the $\rho$-tube.
\end{corollary}

\begin{proof}
By Proposition \ref{prop:family}, $G_\rho$ has isolated critical value and satisfies condition (b) associated with $\rho$; moreover $\dim V_{G_\rho}=2>0$. By Lemma \ref{lem:autoangular}, $\Psi_{G_\rho}$ is $\rho$-regular. Theorem \ref{thm:sphere} and Corollary \ref{cor:tube} conclude.
\end{proof}

\section{The main example}\label{sec:main}

We now turn to the explicit germ described in the introduction. The key is to determine exactly when the projection $\Psi_{G_\rho}$ of a member of the quadratic family acquires critical points along \emph{Euclidean} spheres.

\begin{lemma}\label{lem:angularcrit}
Let $\rho=ax^2+by^2+cz^2$ with $a,b,c>0$, let $G_\rho=(xz,\,yz\rho)$, and let $\rho_E=x^2+y^2+z^2$. With $X=G_1\nabla G_2-G_2\nabla G_1$, one has
\[
X\times\nabla\rho_E=\Bigl(2xz^3\bigl[\rho+2(b-c)y^2\bigr],\ \ 2yz^3\bigl[\rho+2(c-a)x^2\bigr],\ \ 2z^2\bigl[2(a-b)x^2y^2-\rho\,(x^2+y^2)\bigr]\Bigr).
\]
Moreover:
\begin{enumerate}
\item[(i)] $M_{\rho_E}(\Psi_{G_\rho})\setminus V_{G_\rho}\ne\emptyset$ if and only if
\[
c-b>0,\qquad a-c>0,\qquad \Delta:=(a-2c)(c-b)-b(a-c)>0 .
\]
\item[(ii)] In that case, $M_{\rho_E}(\Psi_{G_\rho})\setminus V_{G_\rho}$ is the punctured cone
\[
y^2=\frac{a-c}{c-b}\,x^2,\qquad z^2=\frac{\Delta}{c(c-b)}\,x^2,\qquad x\ne0,
\]
a union of four punctured lines through the origin. Every point of this cone with $z\ne0$ is a critical point of $\Psi_{G_\rho}$ restricted to the Euclidean sphere through it.
\end{enumerate}
\end{lemma}

\begin{proof}
The displayed formula for $X\times\nabla\rho_E$ follows from $X=z^2\bigl[2xy(ax,by,cz)+\rho(-y,x,0)\bigr]$ (as in Lemma \ref{lem:autoangular}) and a direct cross-product computation with $\nabla\rho_E=2(x,y,z)$, which we have also verified independently by symbolic computation.

Points off $V_{G_\rho}$ have $z\ne0$ and $(x,y)\ne(0,0)$. If $x=0$ and $y\ne0$, the second component reduces to $2yz^3\rho\ne0$; if $y=0$ and $x\ne0$, the first reduces to $2xz^3\rho\ne0$. Hence critical points off $V_{G_\rho}$ must have $xyz\ne0$, and the vanishing of the first two components is equivalent to
\begin{equation}\label{eq:angsys}
\rho=2(c-b)\,y^2,\qquad \rho=2(a-c)\,x^2.
\end{equation}
Since $\rho>0$ off the origin and $x,y\ne0$, the two equations in \eqref{eq:angsys} read $c-b=\rho/(2y^2)$ and $a-c=\rho/(2x^2)$, and therefore force
\[
c-b>0\qquad\text{and}\qquad a-c>0
\]
outright; no other sign pattern is possible. Solving the linear system \eqref{eq:angsys} together with $\rho=ax^2+by^2+cz^2$ yields
\[
y^2=\frac{a-c}{c-b}\,x^2,\qquad cz^2=\frac{(a-2c)(c-b)-b(a-c)}{c-b}\,x^2=\frac{\Delta}{c-b}\,x^2,
\]
and existence of a solution with $z\ne0$ requires precisely $\Delta/(c-b)>0$, i.e.\ the sign condition of (i). It remains to check that on this locus the third component of $X\times\nabla\rho_E$ vanishes automatically: substituting $\rho=2(a-c)x^2$ and $y^2=\frac{a-c}{c-b}x^2$ into $2(a-b)x^2y^2-\rho(x^2+y^2)$ gives
\[
2(a-c)x^4\left[\frac{a-b}{c-b}-1-\frac{a-c}{c-b}\right]
=2(a-c)x^4\cdot\frac{(a-b)-(c-b)-(a-c)}{c-b}=0.
\]
 Conversely, if the sign condition fails, the system \eqref{eq:angsys} has no solution with $xyz\ne0$, so $M_{\rho_E}(\Psi_{G_\rho})\setminus V_{G_\rho}=\emptyset$. The last claim of (ii) is Lemma \ref{lem:p2}.
\end{proof}

\begin{remark}\label{rem:nomirror}
One might expect the criterion of Lemma \ref{lem:angularcrit}(i) to be stated symmetrically, as ``$c-b$, $a-c$ and $\Delta$ are nonzero and share a common sign''. That formulation is in fact equivalent, because for $a,b,c>0$ the all-negative pattern never occurs: if $c-b<0$ and $a-c<0$ then $a<c<b$, so $a-2c<-c<0$, whence $(a-2c)(c-b)>0$ and $-b(a-c)>0$, and therefore $\Delta>0$. We state the criterion in the explicit positive form, which is what all the applications below use.
\end{remark}

\begin{theorem}\label{thm:main}
Let
\[
G(x,y,z)=\bigl(xz,\;yz\,(10x^2+y^2+3z^2)\bigr),\qquad \rho(x,y,z)=10x^2+y^2+3z^2.
\]
Then:
\begin{enumerate}
\item[(i)] $G$ admits no Milnor fibration on the Euclidean tube, and no Milnor fibration on the Euclidean sphere, for any radius $\varepsilon>0$. More precisely, for every $\varepsilon>0$ the projection $\Psi_G$ has exactly eight critical points on $S^2_{\rho_E,\varepsilon}\setminus K_{\rho_E,\varepsilon}$, namely the intersections of the sphere with the four lines spanned by $(\pm\sqrt6,\pm\sqrt{21},1)$.
\item[(ii)] $G$ admits both the Milnor fibration on the $\rho$-tube and the Milnor fibration on the $\rho$-sphere.
\end{enumerate}
\end{theorem}

\begin{proof}
Here $(a,b,c)=(10,1,3)$.

\emph{(i), the tube.} In the notation of Theorem \ref{thm:criterion}, take $\rho_1=\rho$ (so $a_1=10$, $b_1=1$, $c_1=3$) and $\rho_2=\rho_E$ (so $a_2=b_2=c_2=1$). Then $\lambda=a_1b_2/(a_2b_1)=10>9$, and Theorem \ref{thm:criterion}(ii) shows both that condition (b) fails for $\rho_E$ and that $G$ admits no Milnor fibration on the Euclidean tube.

\emph{(i), the sphere.} We have
\[
c-b=2>0,\qquad a-c=7>0,\qquad \Delta=(a-2c)(c-b)-b(a-c)=4\cdot2-7=1>0,
\]
so the sign condition of Lemma \ref{lem:angularcrit}(i) holds, and the critical cone of Lemma \ref{lem:angularcrit}(ii) is
\[
y^2=\frac72\,x^2,\qquad z^2=\frac{1}{3\cdot2}\,x^2=\frac{x^2}{6},
\]
i.e.\ the four lines spanned by $(\pm\sqrt6,\pm\sqrt{21},1)$ (normalizing $z=1$: $x^2=6$, $y^2=21$). Fix $\varepsilon>0$ and let $t=\sqrt{\varepsilon/28}$, so that the eight points $p=\pm t(\pm\sqrt6,\pm\sqrt{21},1)$ satisfy $\rho_E(p)=t^2(6+21+1)=\varepsilon$, hence lie on $S^2_{\rho_E,\varepsilon}$. Each such $p$ has $z\ne0$ and $x\ne0$, so $p\notin V_G$, i.e.\ $p\in S^2_{\rho_E,\varepsilon}\setminus K_{\rho_E,\varepsilon}$, and by Lemma \ref{lem:angularcrit}(ii) it is a critical point of $\Psi_G|_{S^2_{\rho_E,\varepsilon}}$. If \eqref{eq:spherefib} (with $\rho=\rho_E$) were a smooth locally trivial fibration, it would be a submersion at every point of its domain, contradicting the existence of $p$. Since $\varepsilon>0$ was arbitrary, no Milnor fibration on the Euclidean sphere exists for any radius. (That these eight points exhaust the critical set on the sphere is the ``only if'' half of Lemma \ref{lem:angularcrit}(i)--(ii).)

\emph{(ii).} This is Corollary \ref{cor:spherefamily} with $(a,b,c)=(10,1,3)$.
\end{proof}

\begin{remark}\label{rem:familyofexamples}
The phenomenon is stable: by Lemma \ref{lem:angularcrit}, every choice of coefficients with $c>b$, $a>c$ and $(a-2c)(c-b)>b(a-c)$ produces a germ $G_\rho=(xz,yz(ax^2+by^2+cz^2))$ whose projection has critical points on every Euclidean sphere---hence no Euclidean sphere fibration---while Corollary \ref{cor:spherefamily} guarantees the fibration on the corresponding $\rho$-sphere. For instance $(a,b,c)=(20,1,4)$ also works ($\Delta=12\cdot3-16=20>0$), whereas $(a,b,c)=(5,1,2)$ does not ($\Delta=1\cdot1-3=-2$, of mixed sign with $c-b=1>0$).
\end{remark}

\begin{remark}\label{rem:whyten}
The coefficient pattern matters in an essential way. For $b=c$ (in particular for the germ $(xz,yz(9x^2+y^2+z^2))$ of Example \ref{ex:nine}), the first component of $X\times\nabla\rho_E$ in Lemma \ref{lem:angularcrit} reduces to $2xz^3\rho$, which forces $x=0$ and then $y=0$: the projection $\Psi_G$ is \emph{submersive} along every Euclidean sphere off the link. For such germs, the failure of condition (b) on the Euclidean tube (Remark \ref{rem:threshold}) does not by itself decide the existence of the Euclidean sphere fibration; it only makes the sufficient criterion of Theorem \ref{thm:sphere} inapplicable. The germ of Theorem \ref{thm:main} is designed precisely so that the angular obstruction is genuine: taking $c$ strictly between $b$ and $a$, with $\Delta>0$, creates the critical cone that rules out the Euclidean sphere fibration outright, while the intrinsic control function $\rho$ removes both the tube and the angular obstructions simultaneously.
\end{remark}

\subsection*{Acknowledgements}
M.~Ribeiro acknowledges financial support from the CNPq-Universal project, MCTI-Grant No.~408147/2023-7, and from FAPES-Grant No.~1021/2025-P:~2025-S37BJ and FAPES-Grant No.~1533/2025-P:~2025-ZTMZN. G.~Fran\c{c}a acknowledges financial support from CAPES-Grant  and FAPES-Grant No.~1533/2025-P:~2025-ZTMZN.


\begin{thebibliography}{99}

\bibitem{ACT13} R.~N.~Ara\'ujo dos Santos, Y.~Chen, M.~Tib\u{a}r, \emph{Singular open book structures from real mappings}, Cent. Eur. J. Math. \textbf{11} (2013), 817--828.

\bibitem{ART10} R.~N.~Ara\'ujo dos Santos, M.~Tib\u{a}r, \emph{Real map germs and higher open book structures}, Geom. Dedicata \textbf{147} (2010), 177--185.

\bibitem{ARSRT19} R.~N.~Ara\'ujo dos Santos, M.~F.~Ribeiro, M.~Tib\u{a}r, \emph{Fibrations of highly singular map germs}, Bull. Sci. Math. \textbf{155} (2019), 92--111.

\bibitem{ARSRT20} R.~N.~Ara\'ujo dos Santos, M.~F.~Ribeiro, M.~Tib\u{a}r, \emph{Milnor--Hamm sphere fibrations and the equivalence problem}, J. Math. Soc. Japan \textbf{72} (2020), 945--957.

\bibitem{Bek91} K.~Bekka, \emph{$(c)$-r\'egularit\'e et trivialit\'e topologique}, in: Singularity Theory and its Applications, Part I, Lecture Notes in Math. 1462, Springer, Berlin, 1991, pp. 42--62.

\bibitem{CMSS09} J.~L.~Cisneros-Molina, J.~Seade, J.~Snoussi, \emph{Refinements of Milnor's fibration theorem for complex singularities}, Adv. Math. \textbf{222} (2009), 937--970.

\bibitem{CSS10} J.~L.~Cisneros-Molina, J.~Seade, J.~Snoussi, \emph{Milnor fibrations and $d$-regularity for real analytic singularities}, Internat. J. Math. \textbf{21} (2010), 419--434.

\bibitem{CSS12} J.~L.~Cisneros-Molina, J.~Seade, J.~Snoussi, \emph{Milnor fibrations and the concept of $d$-regularity for analytic map germs}, Contemp. Math. \textbf{569} (2012), 1--28.

\bibitem{Le77} L\^e D\~ung Tr\'ang, \emph{Some remarks on relative monodromy}, in: Real and Complex Singularities (Oslo, 1976), Sijthoff and Noordhoff, 1977, pp. 397--403.

\bibitem{Loo84} E.~Looijenga, \emph{Isolated Singular Points on Complete Intersections}, London Math. Soc. Lecture Note Ser. 77, Cambridge Univ. Press, 1984.

\bibitem{Massey10} D.~B.~Massey, \emph{Real analytic Milnor fibrations and a strong \L ojasiewicz inequality}, in: Real and Complex Singularities, London Math. Soc. Lecture Note Ser. 380, Cambridge Univ. Press, 2010, pp. 268--292.

\bibitem{Mather12} J.~Mather, \emph{Notes on topological stability}, Bull. Amer. Math. Soc. \textbf{49} (2012), 475--506.

\bibitem{Mil68} J.~W.~Milnor, \emph{Singular Points of Complex Hypersurfaces}, Ann. of Math. Stud. 61, Princeton Univ. Press, 1968.

\bibitem{Pichon05} A.~Pichon, \emph{Real analytic germs $f\bar g$ and open-book decompositions of the $3$-sphere}, Internat. J. Math. \textbf{16} (2005), 1--12.

\bibitem{PichonSeade08} A.~Pichon, J.~Seade, \emph{Fibred multilinks and singularities $f\bar g$}, Math. Ann. \textbf{342} (2008), 487--514.

\bibitem{Quick22} R.~N.~Ara\'ujo dos Santos, D.~Dreibelbis, A.~A.~do Esp\'irito Santo, M.~F.~Ribeiro, \emph{A quick trip through fibration structures}, J. Singul. \textbf{22} (2022), 134--158.

\bibitem{RS05} M.~A.~S.~Ruas, R.~N.~Ara\'ujo dos Santos, \emph{Real Milnor fibrations and $(C)$-regularity}, Manuscripta Math. \textbf{117} (2005), 207--218.

\bibitem{RST24} M.~Ribeiro, I.~Santamaria, T.~da Silva, \emph{Some remarks about $\rho$-regularity for real analytic maps}, Res. Math. Sci. \textbf{11} (2024), Article 40.

\bibitem{RSV02} M.~A.~S.~Ruas, J.~Seade, A.~Verjovsky, \emph{On real singularities with a Milnor fibration}, in: Trends in Mathematics, Birkh\"auser, Basel, 2002, pp. 191--213.

\bibitem{Seade97} J.~Seade, \emph{Open book decompositions associated to holomorphic vector fields}, Bol. Soc. Mat. Mexicana \textbf{3} (1997), 323--336.

\bibitem{Seade19} J.~Seade, \emph{On Milnor's fibration theorem and its offspring after 50 years}, Bull. Amer. Math. Soc. \textbf{56} (2019), 281--348.

\bibitem{Thom69} R.~Thom, \emph{Ensembles et morphismes stratifi\'es}, Bull. Amer. Math. Soc. \textbf{75} (1969), 240--284.

\end{thebibliography}
\end{document}